\documentclass{amsart}

\usepackage{amssymb,latexsym,amsmath,extarrows}
\usepackage{graphicx}

\newtheorem*{acknowledgement}{Acknowledgements}
\newtheorem{theorem}{Theorem}[section]
\newtheorem{lemma}[theorem]{Lemma}
\newtheorem{proposition}[theorem]{Proposition}
\newtheorem{remark}[theorem]{Remark}

\newtheorem{definition}[theorem]{Definition}

\newcommand{\q}{\quad}

\newcommand{\abs}[1]{\lvert#1\rvert}

\newcommand{\e}{\varepsilon}

\newcommand{\wh}{\widehat}

\newcommand{\ZR}{\mathbb{R}}

\newcommand{\ZB}{\mathbb{B}}

\newcommand{\Id}{{\bf 1}}

\newcommand{\bS}{{\bf S}}

\newcommand{\bR}{{\bf R}}

\newcommand{\cQ}{ {\mathcal Q}}
\newcommand{\cV}{ {\mathcal V}}

\newcommand{\bq}{{\bf q}}

\begin{document}

\title{$l^p$ decoupling for restricted $k$-broadness}

\author{Xiumin Du}

\address{
Xiumin Du\\
School of Mathematics\\
Institute for Advanced Study\\
Princeton, NJ, 08540, USA}

\email{xdu@math.ias.edu}

\author{Xiaochun Li}

\address{
Xiaochun Li\\
Department of Mathematics\\
University of Illinois at Urbana-Champaign\\
Urbana, IL, 61801, USA}

\email{xcli@math.uiuc.edu}

\begin{abstract}
In \cite{G2}, to prove Fourier restriction estimate using polynomial partitioning, Guth introduced the concept of $k$-broad part of regular $L^p$ norm and obtained sharp $k$-broad restriction estimates. To go from $k$-broad estimates to regular $L^p$ estimates, Guth employed $l^2$ decoupling result. In this article, similar to the technique introduced by Bourgain-Guth \cite{BG}, we establish an analogue to go from regular $L^p$ norm to its $(m+1)$-broad part, as the error terms we have the restricted $k$-broad parts ($k=2,\cdots,m$). To analyze the restricted $k$-broadness, we prove an $l^p$ decoupling result, which can be applied to handle the error terms and recover Guth's linear restriction estimates in \cite{G2}.  

\end{abstract}

\maketitle

\section{Introduction} \label{intro}
\setcounter{equation}0

In \cite{G2}, to prove Fourier restriction estimate using polynomial partitioning, Guth introduced the concept of $k$-broad part of regular $L^p$ norm and obtained sharp $k$-broad restriction estimates. Applying the $l^2$ decoupling theorem (cf. \cite{B} \cite{BD1}) to go from $k$-broad estimates to regular $L^p$ estimates, Guth obtained new range in \cite {G2} for the Fourier restriction estimate, which improves the result in \cite {BG}. 

In this article, regarding the broadness, we establish an analog of Bourgain-Guth's technique in \cite{BG} to go from multilinear estimates to regular $L^p$ estimates. Going from regular $L^p$ norm to its $(m+1)$-broad part, as the error terms we have the ``restricted'' $k$-broad parts, $k=2,\cdots,m$ (Lemma \ref{lem-rec}). Here by ``restricted'' we mean that functions under consideration are concentrated in small neighborhoods of $k$-dimensional subspaces. To analyze the restricted $k$-broadness, we prove an $l^p$ decoupling result (Theorem \ref{thm-kp}), which can be applied to handle the error terms and recover Guth's linear restriction estimates in \cite{G2} (see Section \ref{appli}). The  proof of our $l^p$ decoupling result uses transverse equi-distribution estimates and the $k$-broad restriction estimates (for some special functions) from \cite{G2}. Therefore, our technique to go from $k$-broad estimates to linear estimates is self-contained, without resorting to $l^2$ decoupling theorem.

\vspace{.15 in}

First, we recall the definition of the $k$-broad part of regular $L^p$-norm from \cite{G2}, and then define a corresponding restricted $k$-broad part.

For any dyadic number $D$, $\cQ_D$ denotes the collection of all dyadic cubes in $\ZR^n$, of
side length $D$. Let $\Sigma$ be the truncated paraboloid,
$$
\Sigma :=\left\{(\xi,\abs \xi^2)\in \ZR^n : \xi \in B^{n-1}(0,1)\right\}\,.
$$
We partition $N_{1/D^2}\Sigma$, the $1/D^2$-neighborhood of $\Sigma$, into disjoint blocks $\theta$ of dimensions $\frac{1}{D}\times\cdots\times \frac{1}{D}\times\frac{1}{D^2}$. Denote this collection of blocks by $\Theta_D$. For each block $\theta$, let $e(\theta)$ be the unit normal vector to $\Sigma$ at the center of $\theta$.

For any function $F:\ZR^n \rightarrow \mathbb C$ with Fourier support in $N_{1/D^2}\Sigma$, $F_\theta$ represents Fourier restriction of $F$ in $\theta$, that is, $\wh{F_\theta}=\Id_\theta\wh F$. From the partition of unity, we write
\begin{equation}
 F(x)= \sum_{\theta \in \Theta_D}F_\theta(x)\,.
\end{equation}
For any positive integer
$d\leq n$,  let $\mathcal V_{d}$ denote the collection of all $d$-dimensional subspaces
in $\ZR^n$. For any $V\in \cV_d$, let $\Theta_{D,V}$ denote the collection of blocks near $V$:
\begin{equation}
 \Theta_{D, V}:= \big\{\theta\in \Theta_D:    {\rm dist}\big(e(\theta), V\big)\leq D^{-1}  \big\}\,
\end{equation}
and define
\begin{equation}
F_{D,V}:=\sum_{\theta\in \Theta_{D,V}} F_\theta\,.
\end{equation}
For any $V_1, \cdots, V_A\in \cV_d$, let $\Theta_D(V_1, \cdots, V_A)$ denote the collection of blocks away from $V_1, \cdots, V_A$:
\begin{equation}
 \Theta_D(V_1, \cdots, V_A):=\big\{\theta\in \Theta_D:
{\rm dist}(e(\theta), V_a) > D^{-1} \,\,{\rm for}\,\, {\rm all}\,\, a\in\{1, \cdots, A\}    \big  \}\,.
\end{equation}

With the notations above, we recall the definition of $k$-broad part of regular $L^p$ norm from \cite{G2}:
\begin{definition}
For any dyadic number $D$, any Schwartz function
$F$ on $\ZR^n$ with Fourier support in $N_{1/D^2}\Sigma$, any integer $2\leq k\leq n$, any positive integer $A$, and any subset $U$ in $\ZR^n$, we define the \textbf{$k$-broad part} of $\|F\|_{L^p(U)}$ with the parameters $A$ and $D$ by
\begin{equation}\label{def-Brk}
 \big\|F\big\|_{BL^p_{k, A, D}(U)}:=
\bigg( \sum_{\bq\in \cQ_{D^2}: \bq \subset U} \min_{V_1, \cdots, V_A\in \cV_{k-1}}
\max_{\theta\in \Theta_{D}(V_1, \cdots, V_A)} \int_{\bq} \big|
  F_{\theta} (x)
\big|^p \, dx \bigg)^{1/p}\,.
\end{equation}
\end{definition}

Now we are ready to define the restricted $k$-broad parts, which arise naturally as error terms when going from regular $L^p$-norm to its higher-order broad part (see Section \ref{sec:lin-broad}). 
For any dyadic number $D$, $D_0$ stands for
the largest dyadic number not exceeding $D^{\e^{\sqrt{L}/2}}$, where $L$ is a large constant and $\e$ is a tiny positive number. 
\begin{definition}
For any dyadic number $D$, any Schwartz function $F$ on $\ZR^n$ with Fourier support in $N_{1/D^2}\Sigma$, any integer $2\leq k\leq n$, any positive integer $A$,
any subset $U$ in $\ZR^n$, and any $k$-dimensional subspace $V\in \cV_k$,
we call the following term
\begin{equation}
\big\|F_{D,V}\big\|_{BL^p_{k, A, D_0}(U)}
\end{equation}
the \textbf{restricted
$k$-broad part} of $\|F\|_{L^p(U)}$ to $V$ with the parameters $A$, $D$ and $D_0$.
\end{definition}

\begin{remark}
We can write out the restricted $k$-broad part explicitly:
\begin{align*}
 &\big\|F_{D,V}\big\|_{BL^p_{k, A, D_0}(U)}\\
 := &\bigg( \sum_{\bq\in \cQ_{D_0^2}: \bq \subset U} \min_{V_1, \cdots, V_A\in \cV_{k-1}}
\max_{\theta\in \Theta_{D_0}(V_1, \cdots, V_A)} \int_{\bq} \big|
  \sum_{\theta'\in \Theta_{D,V}:\theta'\subset \theta} F_{\theta'} (x)
\big|^p \, dx \bigg)^{1/p}\,.
\end{align*}
\end{remark}

It is straightforward to get
\begin{equation}\label{semi}
 \big\| \sum_{j=1}^N F_j \big\|_{BL^p_{k, A, D}(U)}\leq
   N^C \sum_{j=1}^N \big\| F_j \big\|_{BL^p_{k,A/N,D}(U)}\,,
\end{equation}
where $C$ is an absolute constant.
This can be proved by noticing
\begin{align*}
&\min_{V_1, \cdots, V_A\in \cV_{k-1}}\,
\max_{\theta\in \Theta_D (V_1, \cdots, V_A)} (F_{1,\theta}+F_{2,\theta}) \\
\leq &\min_{V_1, \cdots, V_{A'}\in \cV_{k-1}}\,
\max_{\theta\in \Theta_{D}(V_1, \cdots, V_{A'})}F_{1,\theta}
+
\min_{V_{A'+1}, \cdots, V_A\in \cV_{k-1}}\,
\max_{\theta\in \Theta_{D}(V_{A'+1}, \cdots, V_A)}F_{2,\theta}\,,
\end{align*}
for any positive integer $A'$ with $A'\leq A$, and nonnegative functions $F_{1, \theta}$, $F_{2, \theta}$. The inequality (\ref{semi}) serves as weak Minkowski inequality and henceforth
the broad part behaves similarly like a norm.

\vspace{.15 in}
Our main result is the following:
\begin{theorem}\label{thm-kp}
Suppose that $2\leq k\leq n$ and $2\leq p \leq \frac{2k}{k-1}$. Then
for any $\e>0$, there exist constant $C_\e$ and large number $\bar A(\e)$ such that
\begin{equation}\label{dec-0}
  \| F_{D,V}\|_{BL^p_{k, A,D_0}(B)}\leq C_\e D^{(k-1)(\frac 1 2 -\frac 1 p)+\e}
\left( \sum_{\theta\in \Theta_{D,V}}\big\| F_\theta \big\|_{L^p(B)}^p \right)^{1/p}
\end{equation}
holds for any dyadic number $D\geq 1$, any function $F$ on $\ZR^n$ with Fourier support in $N_{1/D^2}\Sigma$, any dyadic cube $B\in\cQ_{D^2}$ in $\ZR^n$, any $k$-dimensional subspace $V\in \cV_k$, and any integer $A\geq \bar A (\e)$.

\end{theorem}

The exponent $(k-1)(\frac 1 2 -\frac 1 p)$ in Theorem \ref{thm-kp} is sharp. To see this, we consider an example where, for a $k$-dimensional subspace $V$, there are $D^{k-1}$ blocks $\theta$ in $\Theta_{D,V}$, and for each block $\theta$, $|F_{\theta}(x)|\sim 1$ on $B\in\cQ_{D^2}$. Since only a tiny fraction of blocks lie near to any $(k-1)$-dimensional subspace, the $k$-broad part is equivalent to regular $L^p$-norm.
Using random sign argument, we can also arrange that $$|F_{D,V}(x)| \sim (\sum_{\theta\in \Theta_{D,V}} |F_\theta(x)|^2)^{1/2}\sim D^{(k-1)/2}$$ for most points $x\in B$. In this scenario, 
$$
\| F_{D,V}\|_{BL^p_{k, A,D_0}(B)} \sim \|F_{D,V}\|_{L^p(B)} \sim D^{\frac{k-1}{2}} D^{\frac{2n}{p}}\,,
$$
$$ \big(\sum_{\theta\in\Theta_{D,V}}\|F_\theta\|^p_{L^P(B^*)}\big)^{1/p}
\sim D^{\frac{k-1}{p}} D^{\frac{2n}{p}}\,.
$$
and therefore,
$$
\frac{\| F_{D,V}\|_{BL^p_{k, A,D_0}(B)}}{\big(\sum_{\theta\in\Theta_{D,V}}\|F_\theta\|^p_{L^P(B^*)}\big)^{1/p}}
\sim D^{(k-1)(\frac 1 2 -\frac 1 p)}\,.
$$
(This example is the same as the example in \cite{G2} showing the sharpness of the $k$-broad estimates.)

In Section \ref{sec:lin-broad}, we prove a lemma which gives the restricted $k$-broad parts ($k=2,\cdots,m$) as the error terms going from regular $L^p$ norm to its $(m+1)$-broad part. In Section \ref{WPD} we review wave packet decomposition. We prove Theorem \ref{thm-kp} in Section \ref{L2} using transverse equi-distribution estimates and the $k$-broad restriction estimates (for some special functions) from \cite{G2}. In Section \ref{appli}, as an application, we show how to go from $k$-broad restriction estimates to linear restriction estimates using our $l^p$ decoupling result.

\vspace{.15in}
\noindent\textbf{List of Notations}:

We use $\e$ to denote a tiny positive number and $L$ to stand for a large constant such that $\e^{L/10}C\leq 1$ for the constant $C$ in 
(\ref{eq-k-k+1}). For any dyadic number $D$, $D_0$ stands for the largest dyadic number not exceeding $D^{\e^{\sqrt L/2}}$.
$C_1\ll C_2$ means that $C_1$ is much smaller than $C_2$;
$C_1\lesssim C_2$ represents that there is an unimportant constant $C$ such that $C_1\leq C C_2$;
and when proving some bound concerning large dyadic number $D$, we write $C_1\lessapprox C_2$ if $C_1 \leq C_\e D^{\e} C_2$ for any $\e>0$.

\section{Passing regular $L^p$ norm to its higher-order broad part} \label{sec:lin-broad}
\setcounter{equation}0

In this section, we get the restricted $k$-broad parts $(k=2,\cdots,m)$ as the error terms going from regular $L^p$-norm to its $(m+1)$-broad part (see Lemma \ref{lem-rec}). The argument used here is in the same spirit as Bourgain-Guth multilinear method in \cite{BG}.

\begin{lemma}\label{lem-k-k+1}
Suppose that $D$ is any dyadic number, $D_0$ is the largest dyadic number not exceeding $D^{\e^{\sqrt L/2}}$, $U$ is a dyadic cube in $\ZR^n$ with side length
larger than $D^2$,
and $2\leq M$, $A$ are integers such that $M^2\leq A$.
Then for $2\leq k\leq n-1$ and any Schwartz function $F$ on $\ZR^n$ with Fourier support
in $N_{1/D^2}\Sigma$, we have
\begin{align} \label{eq-k-k+1}
  \big\| F \big\|_{BL^p_{k, A, D_0}(U)} \leq
   &D^{n}  \big\| F \big\|_{BL^p_{k+1, \frac{ A}{M}, D}(U)} \\
   +  &\big(\frac{A}{M}\big)^C \bigg(      \sum_{B\in \cQ_{D^2}: B\subset U } \sup_{V\in \cV_k}
  \big\|F_{D,V}\big\|^p_{BL^p_{k,M/2,D_0 }(B)}       \bigg)^{\frac{1}{p}} \,, \notag
\end{align}
where $C$ is an absolute constant.
\end{lemma}

\begin{proof}
By the definition \eqref{def-Brk}, we write $\big\| F \big\|_{BL^p_{k, A, D_0}(U)}$ as
\begin{equation} \label{LHS}
\left(\sum_{B\in\cQ_{D^2}:B\subset U}\big\| F \big\|_{BL^p_{k, A, D_0}(B)}^p\right)^{1/p}\,,
\end{equation}
and for each $B\in\cQ_{D^2}$, take subspaces $V_1', \cdots, V_{A/M}'\in \cV_k$ which depend on $B$ and $F$ to be the minimizers obeying
\begin{equation} \label{min}
 \max_{\theta\in \Theta_{D}(V_1', \cdots, V_{A/M}')} \int_B |F_{\theta}|^p
=\min_{V_1, \cdots, V_{A/M}\in\cV_k}\max_{\theta\in \Theta_{D}(V_1, \cdots, V_{A/M})}
\int_B |F_\theta|^p \,.
\end{equation}
On each $B$, we apply \eqref{semi} to the function
$$F=\sum_{\theta \in \Theta_{D}(V_1', \cdots, V_{A/M}')}F_\theta + 
\sum_{\theta \notin \Theta_{D}(V_1', \cdots, V_{A/M}')}F_\theta $$ 
and bound \eqref{LHS} by
\begin{align} 
&\left(\sum_{B\in\cQ_{D^2}:B\subset U}\big\| \sum_{\theta \in \Theta_{D}(V_1', \cdots, V_{A/M}')}F_\theta \big\|_{BL^p_{k, A/2, D_0}(B)}^p\right)^{1/p} \label{LHS1} \\
+ &\left(\sum_{B\in\cQ_{D^2}:B\subset U}\big\| \sum_{\theta \notin \Theta_{D}(V_1', \cdots, V_{A/M}')}F_\theta \big\|_{BL^p_{k, A/2, D_0}(B)}^p\right)^{1/p} \,.\label{LHS2}
\end{align}
For the first term \eqref{LHS1}, note that
$$
\big\| \sum_{\theta \in \Theta_{D}(V_1', \cdots, V_{A/M}')}F_\theta \big\|_{BL^p_{k, A/2, D_0}(B)}
\leq
\big\|
\sum_{\theta \in \Theta_{D}(V_1', \cdots, V_{A/M}')}|F_\theta|
\big\|_{L^p(B)}\,,
$$
therefore, \eqref{LHS1} is bounded by
$$
D^{n} \left(\sum_{B\in\cQ_{D^2}:B\subset U} \max_{\theta \in \Theta_{D}(V_1', \cdots, V_{A/M}')}\int_{B}   |F_\theta|^p \right)^{1/p}\,,
$$
and by the choice as in \eqref{min} and the definition \eqref{def-Brk}, this is
$$D^{n}  \big\| F \big\|_{BL^p_{k+1, \frac{ A}{M}, D}(U)}\,.$$
For the second term \eqref{LHS2}, we have 
$$
  \bigg|\sum_{\theta \notin \Theta_{D}(V_1', \cdots, V_{A/M}')}F_\theta\bigg| \leq \sum_{j=1}^{A/M}\bigg|\sum_{\theta \in \Theta_{D,V'_j}}F_\theta\bigg|
  =
  \sum_{j=1}^{A/M}\bigg|F_{D,V'_j}\bigg|
  \,,
$$
therefore, by the weak triangle inequality \eqref{semi}, the second term \eqref{LHS2} is controlled by
$$
  \left(\frac A M\right)^C \bigg(      \sum_{B\in \cQ_{D^2}: B\subset U } \sup_{V\in \cV_k}
  \big\|F_{D,V}\big\|^p_{BL^p_{k,M/2,D_0 }(B)}       \bigg)^{\frac{1}{p}}\,,
$$
and this completes the proof.
\end{proof}

\begin{lemma}\label{lem-p-B2}
There is a constant C such that for any positive integer $A$, any $R\geq 1$, any Schwartz function $F$ on $\ZR^n$ with Fourier support in $N_{1/R}\Sigma$, any dyadic number $D$ satisfying $ \sqrt R\geq D\geq R^{\e^L}$, and any cube $U\in\cQ_{R}$ in $\ZR^{n}$, the following holds:
\begin{equation}
 \|F\|_{L^p(U)}\leq D^{n}  \big\| F \big\|_{BL^p_{2, A, D}(U)}
+  C A \big\| \max_{\theta\in \Theta_D}\big| F_\theta\big| \big\|_{L^p(U)}\,.
\end{equation}
\end{lemma}

\begin{proof}
This proof is similar to the proof of  Lemma \ref{lem-k-k+1}. We write $\|F\|_{L^p(U)}$ as
$$
 \left( \sum_{\bq\in\cQ_{D^2}:\bq\subset U} \big\|F\big\|_{L^p(\bq)}^p \right)^{1/p}\,,
$$
and for each $\bq\in\cQ_{D^2}$, take subspaces $V_1', \cdots, V_{A}'\in \cV_1$ which depend on $\bq$ and $F$ to be the minimizers obeying
\begin{equation} \label{min'}
 \max_{\theta\in \Theta_{D}(V_1', \cdots, V_{A}')}\int_\bq |F_\theta|^p
=\min_{V_1, \cdots, V_{A}\in\cV_1}\max_{\theta\in \Theta_{D}(V_1, \cdots, V_{A})}
\int_\bq |F_\theta|^p \,,
\end{equation}
then on each $\bq$ by applying Minkowski inequality to function 
$$F=\sum_{\theta \in \Theta_{D}(V_1', \cdots, V_A')}F_\theta + 
\sum_{\theta \notin \Theta_{D}(V_1', \cdots, V_A')}F_\theta $$
we bound $\|F\|_{L^p(U)}$ by 
\begin{align*}
&\left( \sum_{\bq\in\cQ_{D^2}:\bq\subset U} \big\|\sum_{\theta \in \Theta_{D}(V_1', \cdots, V_A')}F_\theta\big\|_{L^p(\bq)}^p \right)^{1/p} \\
+
&\left( \sum_{\bq\in\cQ_{D^2}:\bq\subset U} \big\|\sum_{\theta \notin \Theta_{D}(V_1', \cdots, V_A')}F_\theta\big\|_{L^p(\bq)}^p \right)^{1/p}\,.
\end{align*}
By the choice as in \eqref{min'}, the first term is bounded by
$$
 D^{n}  \big\| F \big\|_{BL^p_{2, A, D}(U)}\,.
$$
Note that there are only $O(A)$ many $\theta$'s that are not in
$\Theta_{D}(V_1', \cdots, V_{A}')$. Henceforth, the second term is controlled by
$$
CA \big\| \max_{\theta\in \Theta_D}\big| F_\theta\big| \big\|_{L^p(U)}\,,
$$
and the proof is done.
\end{proof}

\begin{lemma}\label{lem-rec}
Let $K_1 \ll K_2\ll \cdots\ll K_n$ be tiny positive powers of $R$, obeying that $K_1$ is the largest dyadic number not exceeding $R^{\e^{L/2}}$ and for every $j\in \{2, \cdots, n\}$,
$K_j$ is the smallest dyadic number exceeding $K_{j-1}^{\e^{-\sqrt{L}/2}}$. Let $A_1>\cdots > A_{n+1}$ be tiny positive powers of $R$ given by $A_1=R^{\e^L}$, $A_j=A_1^{2^{1-j}}$ for $j=2,3,\cdots,n$. Then there is a constant $C$ such that
\begin{align}\label{recur1}
 \|F\|_{L^p(U)}\leq &CA_1\big\| \max_{\theta\in \Theta_{K_1}}\big| F_\theta \big| \big\|_{L^p(U)}
+ K_m^{2n}\|F\|_{BL^p_{m+1, A_m, K_m}(U)}
\\
+ &\sum_{j=2}^m K_j^{\e^{10}} \bigg(      \sum_{B\in \cQ_{K_j^2}: B\subset U } \sup_{V\in \cV_j}
  \big\|F_{K_j,V}\big\|^p_{BL^p_{j,A_j/2,K_{j-1} }(B)} \bigg)^{\frac{1}{p}} \notag
\end{align}
holds for any integer $2\leq m\leq n-1$, any Schwartz function $F$ on $\ZR^n$ with Fourier support in $N_{1/R}\Sigma$, and any cube $U\in\cQ_{R}$ in $\ZR^n$.

\end{lemma}

\begin{proof}
Applying Lemma \ref{lem-p-B2} (with $D=K_1$ and $A=A_1$) we get
\begin{equation}\label{p-2B}
 \|F\|_{L^p(U)}\leq K_1^{n}\|F\|_{BL^p_{2, A_1, K_1}(U)} +
CA_1\big\| \max_{\theta\in \Theta_{K_1}}\big| F_\theta \big| \big\|_{L^p(U)}\,.
\end{equation}
Note that by choice of $L$, we have $A_1^C \ll K_1$ for the constant $C$ in Lemma \ref{lem-k-k+1}. For any $j\in \{2, \cdots, m\}$,
we use Lemma \ref{lem-k-k+1} (with $k=j$, $A=A_{j-1}$, $M= \sqrt {A_{j-1}}=A_j$, $D_0=K_{j-1}$, 
$D=K_j$ )
 to obtain 
\begin{align}\label{recur}
  \big\| F \big\|_{BL^p_{j, A_{j-1}, K_{j-1}}(U)} \leq&
   K_j^{n}  \big\| F \big\|_{BL^p_{j+1, A_{j}, K_j}(U)} \\
   + &K_1  \bigg(      \sum_{B\in \cQ_{K_j^2}: B\subset U } \sup_{V\in \cV_j}
  \big\|F_{K_j,V}\big\|^p_{BL^p_{j,A_j/2,K_{j-1} }(B)} \bigg)^{\frac{1}{p}}   \,. \notag
\end{align}
Since $L$ is a sufficiently large constant, (\ref{p-2B}) and (\ref{recur}) yield
(\ref{recur1}) as desired.
\end{proof}

\section{Wave Packet Decomposition}\label{WPD}
\setcounter{equation}0

Let $F$ be a Schwartz function on $\ZR^n$ with Fourier support in $N_{1/D^2}\Sigma$, where $\Sigma$ is the truncated paraboloid in $\ZR^n$. Recall from Section \ref{intro} that we partition $N_{1/D^2}\Sigma$ into disjoint blocks $\theta$ of dimensions $\frac{1}{D}\times\cdots\times\frac{1}{D}\times\frac{1}{D^2}$, and we write
\begin{equation} \label{eq-theta}
F(x)=\sum_{\theta\in \Theta_D} F_\theta (x)\,,
\end{equation}
where $\wh{F_\theta}=\Id_\theta\wh F$.
For each block $\theta \in \Theta_D$, we take a collection $\bS_\theta$ of its dual boxes  $T$'s to tile $\ZR^n$. Note that the pairwise disjoint boxes $T$'s in $\bS_\theta$ are of dimensions $D\times \cdots \times D \times D^2$, with the long axis in the direction $e(\theta)$, which is the unit normal vector to $\Sigma$ at the center of $\theta$.
By partition of unity,
there are nonnegative Schwartz functions $\phi_T$'s such that
each $\phi_T$ is Fourier supported in the translate of $\theta$ at the origin, and
\begin{align} 
 &\phi_T(x) \leq C (1+{\rm dist}(x, T))^{-100n}\,, \label{phiT} \\
 &\sum_{T\in \bS_\theta}\phi_T(x) =1 \,, \label{sum0}
\end{align}
for any $x\in\mathbb \ZR^n$.
Due to \eqref{phiT}, the function $\phi_T$ is essentially a bump function
supported on $T$ and we have
\begin{equation}\label{orth-p}
 \big|\sum_{T \in \bS_\theta }\phi_T \big|^p
 \sim \sum_{ T \in \bS_\theta } \phi_T^p\,.
\end{equation}
By \eqref{eq-theta} and \eqref{sum0}, we obtain wave packet decomposition
\begin{equation}
 F(x) =  \sum_{\theta\in \Theta_D}\sum_{T\in \bS_\theta} \phi_T(x)F_\theta(x)\,.
\end{equation}
We define and write
\begin{align}
  & \bS_D:=\{(\theta,T)\ \,|\,\theta \in \Theta_D, T\in \bS_\theta \}\,, \q F_{\theta,T}:=\phi_T F_\theta\,, \\
& F(x)=\sum_{(\theta,T)\in \bS_D} F_{\theta,T}(x)\,.  \label{wp-F}
\end{align}
For each $(\theta,T)\in \bS_D$, $F_{\theta,T}$ is a function essentially supported in $T$, with Fourier
support in $2\theta$, thus it can be viewed essentially as constant in $T$. By utilizing Bernstein's inequality, for any $2\leq p<\infty$ we get

\begin{equation}\label{p-2}
 \big\|  F_{\theta,T} \big\|_p \sim  |T|^{1/p-1/2} \big\| F_{\theta,T}\big\|_2
\sim D^{(n+2)(1/p-1/2)} \big\| F_{\theta,T}\big\|_2 \,.
\end{equation}

\section{Proof of $l^p$ decoupling result}\label{L2}
\setcounter{equation}0

First we recall a $k$-broad estimate for functions concentrated in a small neighborhood of a $k$-dimensional subspace (this is a special case of Proposition 8.1 in \cite{G2} with $m=k$):

\begin{lemma} [Guth] \label{lem-G}
Let $2\leq k \leq n$. For any $\e>0$, there exist constant $C(\e)$ and large number $\bar A(\e)$ such that the following holds for any integer $A\geq \bar A (\e)$.
For any dyadic numbers $D$ and $\bar D$ with $D_0\leq \bar D \ll D^\e$,
any Schwartz function $F$ on $\ZR^n$ with Fourier support in $N_{1/D^2}\Sigma$, any $k$-dimensional subspace $V$,
any dyadic cube $B\in \cQ_{D^2}$ and any $2\leq p\leq \frac{2k}{k-1}$,
\begin{equation}\label{L2-B}
   \| F_{D,V}\|_{BL^p_{k, A,\bar D}( B)}\leq C_\e  D^{\e^2- (n+k)(\frac12-\frac{1}{p})}\|F_{D,V}\|_{L^2(B)}\,.
\end{equation} 
\end{lemma}

Using Lemma \ref{lem-G} and equi-distribution estimates from \cite{G2}, we obtain the following local estimate for restricted $k$-broad part:

\begin{lemma}
Let $2\leq k\leq n$. For any $\e>0$, there exist constant $C(\e)$ and large number $\bar A(\e)$ such that the following hold for any integer $A\geq \bar A (\e)$.
For any dyadic number $D$,
any Schwartz function $F$ on $\ZR^n$ with Fourier support in $N_{1/D^2}\Sigma$, any $k$-dimensional subspace $V$,
any dyadic cube $Q\in \cQ_{D}$ and any $2\leq p\leq \frac{2k}{k-1}$,
 \begin{equation}\label{p-k-Q}
   \| F_{D,V}\|_{BL^p_{k, A,D_0}(Q)}\leq C_\e D^{\e^2-n(\frac12-\frac{1}{p})}\|F_{D,V}\|_{L^2(Q)}\,.
 \end{equation}
\end{lemma}

\begin{proof}
For $F_{D,V}$ on $Q$, we consider wave packet decomposition $\bS_{\sqrt D}$. Note that for each new wave packet $(\theta,T)\in \bS_{\sqrt D}$, we have 
$$
{\rm{dist}}(e(\theta),V) \lesssim D ^{-1/2}\,.
$$
We partition $Q$ into translates of the $\sqrt D$-neighborhood of $V$, and write
$$
Q=\bigcup_j (Q\cap N_{\sqrt D} V_j) \,.
$$
For each $j$, we consider the function
$$
 F_{\sqrt D, V_j} = \sum_{(\theta,T)\in S_{\sqrt D}: T \cap Q\cap N_{\sqrt D}V_j \neq \emptyset} (F_{D,V})_{\theta,T}\,.
$$
Given $V\in \cV_k$, there is a $(n-k)$-dimensional plane $V'$ which is transversal to $V$ such that the following holds. In a $D$-cube 
$Q$,  because of its Fourier support,  the function $F_{D,V}$ can be viewed essentially as constant along $(n-k)$-dimensional planes parallel to $V'$.
Henceforth we have the following equi-distribution inequality (cf. Lemma 6.2 in \cite{G2}), 
\begin{equation} \label{equi}
\|F_{\sqrt D,V_j}\|_{L^2(Q)}^2 \lessapprox \bigg(\frac {\sqrt{D}}{D} \bigg)^{n-k} \|F_{D,V}\|_{L^2(Q)}^2\,.
\end{equation}
Now we estimate the left side of \eqref{p-k-Q} using Lemma \ref{L2-B}:
\begin{equation*}
  \| F_{D,V}\|_{BL^p_{k, A,D_0}(Q)}^p =
  \sum_j
    \| F_{D,V}\|_{BL^p_{k, A,D_0}(Q\cap N_{\sqrt D}V_j)}^p
\end{equation*}
$$
\sim \sum_j 
    \| F_{\sqrt D,V_j}\|_{BL^p_{k, A,D_0}(Q)}^p
\lessapprox \sum_j \bigg[ 
D^{ -\frac 12 (n+k)(\frac 1 2-\frac 1p)}
\bigg]^p \|F_{\sqrt D,V_j}\|_{L^2(Q)}^p
$$
and by orthogonality and the equi-distribution \eqref{equi}, this is further bounded by
$$
\bigg[ 
D^{-\frac 12 (n+k)(\frac 1 2-\frac 1p)}
\bigg]^p
\sum_j  \|F_{\sqrt D,V_j}\|_{L^2(Q)}^2 \bigg(  \bigg(\frac {\sqrt{D}}{D} \bigg)^{n-k} \|F_{ D,V}\|^2_{L^2(Q)} \bigg)^{\frac{p-2}{2}}
$$
$$
\sim \bigg[
D^{-n(\frac 1 2-\frac 1p)}  \|F_{ D,V}\|_{L^2(Q)} 
\bigg]^p\,,
$$
as desired.
\end{proof}

Now we are ready to prove the $l^p$ decoupling result in Theorem \ref{thm-kp}:
\begin{proposition}
Suppose that $2\leq k\leq n$ and $2\leq p \leq \frac{2k}{k-1}$. Then
for any $\e>0$, there exist constant $C_\e$ and large number $\bar A(\e)$ such that
\begin{equation}\label{dec-1}
  \| F_{D,V}\|_{BL^p_{k, A,D_0}(B)}\leq C_\e D^{(k-1)(\frac 1 2 -\frac 1 p)+\e}
\left( \sum_{\theta\in \Theta_{D,V}}\big\| F_\theta \big\|_{L^p(B)}^p \right)^{1/p}
\end{equation}
holds for any dyadic number $D\geq 1$, any function $F$ on $\ZR^n$ with Fourier support in $N_{1/D^2}\Sigma$, any dyadic cube $B\in\cQ_{D^2}$ in $\ZR^n$, any $k$-dimensional subspace $V\in \cV_k$, and any integer $A\geq \bar A (\e)$.
\end{proposition}

\begin{proof}
By definition, we have
\begin{equation} \label{eq:def}
\|F_{D,V}\|^p_{BL^p_{k, A,D_0}(B)} = \sum_{Q\in\cQ_{D}:Q\subset B}
\big\|F_{D,V}\big\|^p_{BL^p_{k, A,D_0}(Q)}\,,
\end{equation}
by the local estimate \eqref{p-k-Q}, orthogonality and H\"older's inequality, we bound \eqref{eq:def} by
\begin{align*}
&\lesssim D^{\e^2p-n(\frac p2-1)} \sum_{Q\in\cQ_{D}:Q\subset B} \big\|F_{D,V}\big\|^p_{L^2(Q)}\\
&\sim D^{\e^2p-n(\frac p2-1)} \sum_{Q\in\cQ_{D}:Q\subset B} \big(\sum_{\theta\in \Theta_{D,V}} \|F_\theta\|_{L^2(Q)}^2\big)^{p/2}\\
& \lesssim D^{\e^2p} \sum_{Q\in\cQ_{D}:Q\subset B} \big(\sum_{\theta\in \Theta_{D,V}} \|F_\theta\|_{L^p(Q)}^2\big)^{p/2}\,.
\end{align*}
Note that for $V\in\cV_k$, we have $|\Theta_{D,V}|\lesssim D^{k-1}$, therefore the above is further bounded by
\begin{align*}
 &\lesssim D^{\e^2p+(k-1)(\frac p 2 -1)} \sum_{Q\in\cQ_{D}:Q\subset B} \sum_{\theta\in \Theta_{D,V}} \|F_\theta\|_{L^p(Q)}^p\\
 &\sim D^{\e^2p+(k-1)(\frac p 2 -1)} \sum_{\theta\in \Theta_{D,V}} \|F_\theta\|_{L^p(B)}^p\,,
\end{align*}
and the proof is complete by taking the $p$-th root.
\end{proof}

\section{An application of Theorem \ref{thm-kp}}\label{appli}
\setcounter{equation}0

Let $\ZB^{n-1}$ denote the unit ball in $\ZR^{n-1}$. 
For $f:\ZB^{n-1}\rightarrow \mathbb{C}$, $Ef$ represents the Fourier extension operator which is given by
\begin{equation}
 Ef (x)=\int_{\ZB^{n-1}}  e^{ix'\cdot \xi} e^{ix_n|\xi|^2} f(\xi) d\xi\,,
\end{equation}
where $x=(x',x_n)\in \ZR^n$. 

Recall from Lemma \ref{lem-rec} that we have parameters $A_j, K_j, j=1,2,\cdots,n,$ which are tiny positive powers of $R$ satisfying
$$
A_n<\cdots<A_2<A_1\ll K_1\ll K_2 \ll \cdots \ll K_n\,.
$$
In \cite{G2}, Guth proved the following sharp $k$-broad restriction estimates:

\begin{theorem}[$k$-broad restriction, \cite{G2}]\label{G2}
Let $1\leq m\leq n-1$. Then for any $p\geq \frac{2(n+m+1)}{n+m-1}$, for any $\e>0$, there is a constant $C_{p,\e}$ so that the following holds for any $U\in \cQ_{R}$ in $\ZR^n$:
\begin{equation}
 \|Ef\|_{BL^{p}_{m+1, A_m, K_m}(U)}\leq C_{p,\e} R^\e \|f\|_{L^2(\ZB^{n-1})}\,.
\end{equation}
\end{theorem}

Applying the $l^2$ decoupling result in \cite{B} to go from $k$-broad estimates to regular $L^p$ estimates, Guth \cite{G2} established the following:

\begin{theorem}[linear restriction, \cite{G2}]\label{G1}
Let $U\in \cQ_{R}$ be a dyadic cube with side length $R$ in $\ZR^{n}$. Then
\begin{equation}
 \|E f\|_{L^p(U)}\leq C_\e R^\e \|f\|_{L^p(\ZB^{n-1})}\,.
\end{equation}
provided that
\begin{equation}
 p>\frac{2(3n+1)}{3n-3} \,\,\text{for}\,\, n\,\,{\rm odd}\,,
\end{equation}
and
\begin{equation}
  p>\frac{2(3n+2)}{3n-2} \,\,\text{for}\,\, n\,\,{\rm even}\,.
\end{equation} 
\end{theorem}

\vspace{.15in}
In this section, we present that \emph{the $l^p$ decoupling result in Theorem \ref{thm-kp} together with the $k$-broad restriction estimates in Theorem \ref{G2} yield the linear restriction estimates in Theorem \ref{G1}}.

Let $\beta_p\geq \epsilon$ denote the best constant such that
\begin{equation}
 \|Ef \|_{L^p(U)}\leq R^{\beta_p} \|f\|_{L^p(\ZB^{n-1})}\,,
\end{equation}
for any sufficiently large dyadic number $R$ and any cube $U\in\cQ_{R}$ in $\ZR^n$.

Given $U\in \cQ_{R}$ in $\ZR^n$. Note that $\bR f$ localized in $U$ is a function with Fourier support
in $N_{1/R}\Sigma$. Recall that we partition $N_{1/D^2}\Sigma$ into blocks $\theta$'s in $\Theta_D$. For each $\theta\in \Theta_D$, $\rm{Proj}(\theta)$ denotes the projection of $\theta$ onto $\ZR^{n-1}\times \{0\}$,
 and $f_\theta$ denotes the restriction of $f$ to $\rm{Proj}(\theta)$.
Applying (\ref{recur1}) in Lemma \ref{lem-rec}, we control $\|Ef\|_{L^p(U)}$
by
\begin{align} 
&CA_1\big\| \max_{\theta\in \Theta_{K_1}}\big| Ef_\theta \big| \big\|_{L^p(U)}
+ K_m^{2n}\|Ef\|_{BL^p_{m+1, A_m, K_m}(U)} \label{linear-broad} \\
+ &\sum_{j=2}^m K_j^{\e^{10}}
\bigg(   \sum_{B\in \cQ_{K_j^2}: B\subset U }
  \sup_{V\in \cV_j}
  \big\|Ef_{K_j,V}\big\|^p_{BL^p_{j, A_j/2,K_{j-1}}(B)}   \bigg)^{\frac{1}{p}} \,, \notag
\end{align}

\noindent for any $m\in\{2, \cdots, n-1\}$. Here $K_j$'s and $A_j$'s are defined as in Lemma \ref{lem-rec}.

First we use parabolic rescaling and induction on radius to control the first term in \eqref{linear-broad}.
\begin{lemma} \label{lin-max}
Let $p\geq \frac{2n}{n-1}$, and $U\in \cQ_{R}$ in $\ZR^n$. Then
 \begin{equation}
  \big\| \max_{\theta \in \Theta_{K_1}} \big|Ef_\theta \big| \big\|_{L^p(U)}
\lesssim \bigg(\frac{R}{K_1}\bigg)^{\beta_p}\|f\|_{L^p(\ZB^{n-1})}\,.
 \end{equation}
\end{lemma}

\begin{proof}
By parabolic rescaling, we get

\begin{align}
   &\big\| \max_{\theta \in \Theta_{K_1}} \big| Ef_\theta \big| \big\|_{L^p(U)}^p
  \leq
  \sum_{\theta \in \Theta_{K_1}}\|
  Ef_\theta\|^p_{L^p(U)} \label{resc-K1}
\\
\leq
  &\sum_{\theta \in \Theta_{K_1}}
  \bigg(K_1^{\frac{2n}{p}-(n-1)} \big(\frac{R}{K_1}\big)^{\beta_p}\bigg)^p\|f_\theta\|^p_p
  \sim \bigg(K_1^{\frac{2n}{p}-(n-1)} \big(\frac{R}{K_1}\big)^{\beta_p}\|f\|_p\bigg)^p\,, \notag
\end{align}
and this completes the proof since $p\geq \frac{2n}{n-1}$.
\end{proof}

Next we use the $l^p$ decoupling result in Theorem \ref{thm-kp} and parabolic rescaling to control the third term in \eqref{linear-broad}.
\begin{lemma} \label{error-j}
Suppose that $2\leq j \leq n$, and $\frac{2j}{j-1} \geq p \geq \frac{2(n-\frac{j-1}{2})}{n-\frac{j+1}{2}}$. Then for any $U\in \cQ_{R}$ in $\ZR^n$,
we have
\begin{equation}\label{Xi-f}
 \bigg(   \sum_{B\in \cQ_{K_j^2}: B\subset U }
  \sup_{V\in \cV_j}
  \big\|Ef_{K_j,V}\big\|^p_{BL^p_{j, A_j/2,K_{j-1}}(B)}   \bigg)^{\frac{1}{p}}
  \leq C_\e K_j^{\e^2}
 \bigg( \frac{R}{K_j}\bigg)^{\beta_p} \|f\|_p\,.
\end{equation}
\end{lemma}

\begin{proof}

By the decoupling result \eqref{dec-0} in Theorem \ref{thm-kp}, we get
\begin{align*}
&\bigg(   \sum_{B\in \cQ_{K_j^2}: B\subset U }
  \sup_{V\in \cV_j}
  \big\|Ef_{K_j,V}\big\|^p_{BL^p_{j, A_j/2,K_{j-1}}(B)}   \bigg)^{\frac{1}{p}}\,.
\\
\lesssim
 & K_j^{(j-1)(\frac 12 -\frac 1p)+\e^2}
  \bigg(
  \sum_{B\in \cQ_{K_j^2}: B\subset U } 
  \sup_{V\in \cV_j}
  \sum_{\theta\in \Theta_{K_j,V}} \big\|Ef_\theta\big\|_{L^p(B)}^p
 \bigg)^{\frac 1p} \\
 \lesssim & K_j^{(j-1)(\frac 12 -\frac 1p)+\e^2} \bigg( \sum_{\theta\in \Theta_{K_j}} \big\|Ef_\theta\big\|_{L^p(U)}^p
 \bigg)^{\frac 1p}
\end{align*}
and by parabolic rescaling this is bounded by
\begin{align*}
\lesssim &K_j^{(j-1)(\frac 12 -\frac 1p)+\e^2} K_j^{\frac{2n}{p}-(n-1)} \bigg(\frac{R}{K_j}\bigg)^{\beta_p}
\bigg(\sum_{\theta \in \Theta_{K_j}} \big\| f_{\theta}\|_{L^p}^p
  \bigg)^{\frac 1 p}
\\
\lesssim & K_j^{\e^2}
 \bigg( \frac{R}{K_j}\bigg)^{\beta_p} \|f\|_p \,,
\end{align*}
provided that 
$$
(j-1)(\frac 12 -\frac 1p)+\frac{2n}{p}-(n-1) \leq 0\,
$$
that is,
$$
p\geq \frac{2(n-\frac{j-1}{2})}{n-\frac{j+1}{2}}\,.
$$
\end{proof}

\begin{remark}
Note that $\frac{2j}{j-1} \geq \frac{2(n-\frac{j-1}{2})}{n-\frac{j+1}{2}}$ holds if and only if $j \leq \frac{2n+1}{3}$. In our application below, we only care about the cases where $2\leq j \leq n/2$, so the constraint $j \leq \frac{2n+1}{3}$ is acceptable.
\end{remark}

Now we can derive the linear restriction estimate in Theorem \ref{G1}, using the inductive arguments above together with the $k$-broad restriction estimates in Theorem \ref{G2}. Indeed, we prove
\begin{equation}
 \|Ef\|_{L^p(U)}\leq C_\e R^\e \|f\|_{L^p(\ZB^{n-1})}\,.
\end{equation}
by inductions. We bound $\|Ef\|_{L^p(U)}$ by the three terms in \eqref{linear-broad}. By Lemma \ref{lin-max}, we can control the first term in \eqref{linear-broad} by induction, requiring that 
$$
p\geq \frac{2n}{n-1}\,.
$$
And by Lemma \ref{error-j}, we can control the third term in \eqref{linear-broad} by induction, requiring that 
$$
\frac{2j}{j-1} \geq p \geq \frac{2(n-\frac{j-1}{2})}{n-\frac{j+1}{2}} \quad \rm{for}\,\, j=2,\cdots,m\,.
$$
Finally, using the $k$-broad restriction estimates in Theorem \ref{G2}, we can get the desired bound for the second term in \eqref{linear-broad}, for
$$
p\geq \frac{2(n+m+1)}{n+m-1}\,.
$$
In summary, we have the desired bound for $p\geq \frac{2(n+m+1)}{n+m-1}$ provided that
$$
\frac{2(n+m+1)}{n+m-1} \geq \frac{2(n-\frac{m-1}{2})}{n-\frac{m+1}{2}}\,,
$$
that is,
$$
m\leq \frac {n}{2}\,.
$$
By taking $m=(n-1)/2$ when $n$ is odd, and $m=n/2$ when $n$ is even, we obtain the linear restriction estimate in Theorem \ref{G1}.

\begin{acknowledgement}
The first author is supported by the National Science Foundation under Grant No. 1638352, as well as support from the Shing-Shen Chern Foundation. The  authors wish to express their indebtedness to Larry Guth for his hospitality when visiting MIT
and his inspirations in mathematics. 
\end{acknowledgement}

\end{document}